\documentclass[xcolor=x11names,reqno,12pt]{amsart}
 \usepackage{amsmath}
 \usepackage{amsthm}
 \usepackage{amssymb}
 \usepackage{diagbox}
 \usepackage{latexsym,longtable}
 \usepackage{graphicx}
 \usepackage{multicol}
 \usepackage{mathrsfs}
 \usepackage{young}
 \usepackage[vcentermath,enableskew,stdtext]{youngtab}
 \usepackage[left=0.9in,right=0.9in,top=1.08in,bottom=1.1in]{geometry}
 \usepackage[all]{xy}
    \SelectTips{cm}{10}     
    \everyxy={<2.5em,0em>:} 
 \usepackage{fancyhdr}      
 \usepackage{multicol}
 \usepackage{multirow}
 \usepackage{enumitem}
 \usepackage{bm}
 \usepackage{etex}
 \usepackage{tikz}
 \usepackage{float}
 \usepackage{array}
 \usepackage{colortbl}
 \usepackage{mathtools}
 \restylefloat{figure}
\numberwithin{equation}{section}
\usepackage{url}
 \usepackage{ytableau}

\usepackage{hyperref}
\usepackage{centernot}

\makeatletter
\newcommand*{\centerfloat}{%
  \parindent \z@
  \leftskip \z@ \@plus 1fil \@minus \textwidth
  \rightskip\leftskip
  \parfillskip \z@skip}
\makeatother

\theoremstyle{plain}
\newtheorem{theorem}{Theorem}[section]
\newtheorem{lemma}[theorem]{Lemma}
\newtheorem{corollary}[theorem]{Corollary}
\newtheorem{proposition}[theorem]{Proposition}

\newtheorem*{conjecture*}{Conjecture}

\theoremstyle{definition}

\theoremstyle{definition}

\newtheorem*{example*}{Example}

\newcommand{\ignore}[1]{}

\newcommand{\h}{\ensuremath{\mathfrak{h}}}

\newcommand{\be}{\begin{equation}}
\newcommand{\ee}{\end{equation}}

\newcommand{\crc}[1]{#1\star}

\newlength{\mycellsize}
\newcommand\mytbl[1]{
\vcenter{
\let\\=\cr
\baselineskip=-16000pt \lineskiplimit=16000pt \lineskip=0pt
\halign{&\mytblcell{##}\cr#1\crcr}}}

\newcommand{\mytblcell}[1]{{%
\def \arg{#1}\def \void{}%
\ifx \void \arg
\vbox to \mycellsize{\vfil \hrule width \mycellsize height 0pt}%
\else \unitlength=\mycellsize
\begin{picture}(1,1)
\put(0,0){\makebox(1,1){$#1\vphantom{\crc{#1}}$}}
\put(0,0){\line(1,0){1}}
\put(0,1){\line(1,0){1}}
\put(0,0){\line(0,1){1}}
\put(1,0){\line(0,1){1}}
\end{picture}%
\fi}}

\newlength{\cellsize}
\newcommand\mytableau[1]{
\vcenter{
\let\\=\cr
\baselineskip=-16000pt \lineskiplimit=16000pt \lineskip=0pt
\halign{&\mytableaucell{##}\cr#1\crcr}}}

\newcommand{\mytableaucell}[1]{{%
\def \arg{#1}\def \void{}%
\ifx \void \arg
\vbox to \cellsize{\vfil \hrule width \cellsize height 0pt}%
\else \unitlength=\cellsize
\begin{picture}(1,1)
\put(0,0){\makebox(1,1){$#1\vphantom{\crc{#1}}$}}
\put(0,0){\line(1,0){1}}
\put(0,1){\line(1,0){1}}
\put(0,0){\line(0,1){1}}
\put(1,0){\line(0,1){1}}
\end{picture}%
\fi}}

\newcommand\boldtableau[1]{
\vcenter{
\let\\=\cr
\baselineskip=-16000pt \lineskiplimit=16000pt \lineskip=0pt
\halign{&\boldtableaucell{##}\cr#1\crcr}}}

\newcommand{\boldtableaucell}[1]{{%
\def \arg{#1}\def \void{}%
\ifx \void \arg
\vbox to \cellsize{\vfil \hrule width \cellsize height 0pt}%
\else \unitlength=\cellsize
\begin{picture}(1,1)
\put(0,0){\makebox(1,1){$\mathbf{#1\vphantom{\crc{#1}}}$}}
\put(0,0){\line(1,0){1}}
\put(0,1){\line(1,0){1}}
\put(0,0){\line(0,1){1}}
\put(1,0){\line(0,1){1}}
\end{picture}%
\fi}}

\allowdisplaybreaks

\title{Distribution Properties for $t$-hooks in Partitions}

\keywords{}

\begin{document}

\author{William Craig}
\address{Department of Math, University of Virginia, Charlottesville, VA 22904}
\email{wlc3vf@virginia.edu}

\author{Anna Pun}
\address{Department of Math, University of Virginia, Charlottesville, VA 22904}
\email{annapunying@gmail.com}

\begin{abstract}
Partitions, the partition function $p(n)$, and the hook lengths of their Ferrers-Young diagrams are important objects in combinatorics, number theory and representation theory. For positive integers $n$ and $t$, we study $p_t^e(n)$ (resp. $p_t^o(n)$), the number of partitions of $n$ with an even (resp. odd) number of $t$-hooks. We study the limiting behavior of the ratio $p_t^e(n)/p(n)$, which also gives $p_t^o(n)/p(n)$ since $p_t^e(n) + p_t^o(n) = p(n)$. For even $t$, we show that
$$\lim\limits_{n \to \infty} \dfrac{p_t^e(n)}{p(n)} = \dfrac{1}{2},$$ and for odd $t$ we establish the non-uniform distribution $$\lim\limits_{n \to \infty} \dfrac{p^e_t(n)}{p(n)} = \begin{cases} \dfrac{1}{2} + \dfrac{1}{2^{(t+1)/2}} & \text{if } 2 \mid n, \\ \\ \dfrac{1}{2} - \dfrac{1}{2^{(t+1)/2}} & \text{otherwise.} \end{cases}$$ Using the Rademacher circle method, we find an exact formula for $p_t^e(n)$ and $p_t^o(n)$, and this exact formula yields these distribution properties for large $n$. We also show that for sufficiently large $n$, the sign of $p_t^e(n) - p_t^o(n)$ is periodic.
\end{abstract}

\maketitle

\section{Introduction and Statement of Results}

For a positive integer $n$, a \textit{partition} $\lambda = (\lambda_1 \geq \lambda_2 \geq \dots \geq \lambda_k)$ of $n$ is defined as a weakly decreasing sequence of positive integers whose sum is $n$. For every $n$, the partition function $p(n)$ is defined as the number of partitions of $n$. The study of partitions and the partition function has been of historical importance both in combinatorics and number theory. One of the most important tools used in the combinatorial study of partitions is the \textit{Ferrers-Young diagrams}. This is a geometric representation of a partition $(\lambda_1 \geq \dots \geq \lambda_k)$ as a grid of $k$ rows of left-aligned square cells, with row $i$ containing $\lambda_i$ cells. For instance, the partition $(5, 4, 1)$ of 10 has the Ferrers-Young diagram

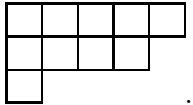
\begin{center}
$\begin{Young} & & & & \cr & & & \cr \cr \end{Young}$.
\end{center}

\noindent One important set of quantities associated with the Ferrers-Young diagram of a partition $\lambda$ are the hook numbers $H_\lambda(i,j)$, the number of cells $(a,b)$ in the diagram of $\lambda$ such that $i \leq a$ and $j \leq b$. It is common to represent all hook numbers of $\lambda$ by placing $H_\lambda(i,j)$ in the interior of each cell of the diagram. For instance, doing so for the partition $\lambda = (5, 4, 1)$ yields

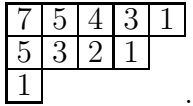
\begin{center}
$\begin{Young} 7 & 5 & 4 & 3 & 1 \cr 5 & 3 & 2 & 1 \cr 1 \cr \end{Young}$.
\end{center}


For a partition $\lambda$ of $n$, a Young tableau of shape $\lambda$ is a labelling of each cell in the Ferrers-Young diagram of shape $\lambda$ by a distinct number from $1$ through $n$. A standard Young tableau is a Young tableau such that entries are increasing in rows (from left to right) and columns (from top to bottom). These tableau are central objects in the representation theory of $S_n$ \cite{Young1902,GA1981,Robinson61}. For example, the irreducible representations of $S_n$ are in one-to-one correspondence with the partitions of $n$ and given a partition $\lambda$, and the degree $f^\lambda$ of the irreducible representation of $S_n$ corresponding to $\lambda$ is the number of the standard Young tableaux of shape $\lambda$. The hook numbers play an important role in calculating $f^\lambda$ via the Frame-Thrall-Robinson hook length formula \cite{FRT54}:
$$f^\lambda = \frac{n!}{\prod\limits_{(i,j)} H_\lambda(i,j)}.$$
         
Hook numbers and related quantities arise in other ways in combinatorics and representation theory. The multiset of hook numbers of $\lambda$, denoted $\mathcal{H}(\lambda)$, appears in the famous Nekrosov-Okounkov formula (formula (6.12) in \cite{NO2006}) $$\sum_{\lambda \in \mathcal P} x^{|\lambda|} \prod_{h \in \mathcal{H}(\lambda)} (1 - z/h^2) = \prod_{k \geq 1} (1 - x^k)^{z-1},$$ which represents an extraordinary generalization of identities of Euler and Jacobi. Restricting to hook numbers that are multiples of a fixed $t$ is natural and leads to the study of $t$-core partitions, which are defined as partitions for which no hook number is a multiple of $t$. The numbers of $t$-core partitions have been the subject of much research, including the famous result of Granville and Ono proving that $t$-core partitions exist for all positive integers when $t \geq 4$ \cite{GO1996}. The study of $t$-cores is fundamental to the modular representation theory of symmetric groups. \\

For the positive integer $t \geq 2$ and $\lambda$ a partition of any positive integer, we wish to study the multiset $\mathcal{H}_t(\lambda)$ of hook numbers which are multiples of $t$. In particular, we are interested in the parity of $\# \mathcal{H}_t(\lambda)$. To study this quantity, we define the partition functions $p^e_t(n)$ and $p^o_t(n)$ to count the number of partitions $\lambda$ of $n$ for which $\# \mathcal{H}_t(\lambda)$ is even and odd, respectively. That is, we define
\begin{equation} \label{Even/Odd Definition}
\begin{split}
p^e_t(n) := \# \{ \lambda \vdash n : \# \mathcal{H}_t(\lambda) \equiv 0 \pmod{2} \}, \\ p^o_t(n) := \# \{ \lambda \vdash n : \# \mathcal{H}_t(\lambda) \equiv 1 \pmod{2} \}.
\end{split}
\end{equation}
Of particular interest here is the distribution of the parity of $\# \mathcal{H}_t(\lambda)$. To study this distribution, we define the functions $\delta^e_t(n) := \dfrac{p^e_t(n)}{p(n)}$ and $\delta^o_t(n) := \dfrac{p^o_t(n)}{p(n)}$. There have been recent developments in studying distribution properties that make use of the properties of $t$-hooks. For example, Peluse has proved that the density of odd values in the character table of $S_n$ goes to zero as $n \to \infty$ \cite{Peluse}. We perform a similar distribution analysis for the parity of the number of $t$-hooks.

By definition, we have $p^e_t(n) + p^o_t(n) = p(n)$, and so $\delta_t^e(n) + \delta_t^o(n) = 1$. Naively, one would expect an even distribution of parities as $n$ becomes large, that is, we would expect that $\delta^e_t(n) \to 1/2$ and $\delta_t^o(n) \to 1/2$. Numerically, this initial speculation receives support for small values of $t$ like $t = 2, 4, 6$, and $8$, as the following table suggests.

\begin{table}[h]
\begin{tabular}{|c|c|c|c|c|c|} \hline
$t$ & $\delta_t^e(100)$ & $\delta_t^e(1000)$ & $\delta_t^e(10000)$ & $\cdots$ & $\infty$ \\ \hline
2 & 0.56611246  &0.50027931  & 0.50000000  &$\cdots$ & $\frac{1}{2}$ \\ \hline
4 & 0.47067843  &0.50002869  & 0.50000000  &$\cdots$ & $\frac{1}{2}$ \\ \hline
6 & 0.52465920  &0.50007471  & 0.50000000  &$\cdots$ & $\frac{1}{2}$ \\ \hline
8 & 0.49484348  &0.49999135  & 0.50000000  &$\cdots$ & $\frac{1}{2}$ \\ \hline
\end{tabular}
\caption{Data for $\delta_t^e(n)$, even $t$}
\end{table}
However, numerical evidence below for the cases $t = 3 ,5 ,7$ and $9$ appears to refute this naive guess. In fact, these tables suggest the existence of multiple limiting values. 

\begin{table}[h]
\begin{tabular}{|c|c|c|c|c|c|c|} \hline
$t$ & $\delta^e_t(100)$ & $\delta^e_t(500)$ & $\delta^e_t(1000)$ & $\delta_t^e(1500)$  &$\cdots$ & $\infty$ \\ \hline
3 &0.7137967695 & 0.7502983017 &0.7499480195 & 0.7500039425 & $\cdots$ & $\frac{3}{4}$ \\ \hline
5 &0.6374948698 & 0.6252149479 &0.6250102246 & 0.6250009877 &  $\cdots$ & $\frac{5}{8}$ \\ \hline
7 &0.5468769228 & 0.5624965413 &0.5625165550 & 0.5624989487 &  $\cdots$ & $\frac{9}{16}$ \\ \hline
9 &0.5375271584 & 0.5313027269 &0.5312496766 & 0.5312499631 &  $\cdots$ & $\frac{17}{32}$ \\ \hline
\end{tabular}
\caption{Data for $\delta_t^e(n)$, $t$ odd and $n$ even.}
\end{table}

\begin{table}[h]
	\begin{tabular}{|c|c|c|c|c|c|c|} \hline
		$t$ & $\delta^e_t(101)$ & $\delta^e_t(501)$ & $\delta^e_t(1001)$ & $\delta_t^e(1501)$ &  $\cdots$ & $\infty$ \\ \hline
		3 &0.2376157284 & 0.2494431573 & 0.2499820335 & 0.2500060167 &  $\cdots$ & $\frac{1}{4}$ \\ \hline
		5 &0.3755477486 & 0.3750000806 & 0.3750000001 & 0.3750000000 &  $\cdots$ & $\frac{3}{8}$ \\ \hline
		7 &0.4396942088 & 0.4374987794 & 0.4374959329 & 0.4375000006 &  $\cdots$ & $\frac{7}{16}$ \\ \hline
		9 &0.4787668076 & 0.4688094755 & 0.4687535414 & 0.4687510507 &  $\cdots$ & $\frac{15}{32}$ \\ \hline
	\end{tabular}
	\caption{Data for $\delta_t^e(n)$, $t$ odd and $n$ odd.}
\end{table}

In this paper, we prove two main theorems which explain this data, and offer exact values on the limiting values in these distributions.

\begin{theorem} \label{Even and Odd T Behavior}
Assuming the notation above, the following are true.

1) If $t>1$ is an even integer, then $$\lim_{n \to \infty} \delta_t^e(n) = \lim_{n \to \infty} \delta_t^o(n) = \dfrac{1}{2}.$$

2) If $t>1$ is an odd integer, then we have $$\lim_{n \to \infty} \delta_t^e(n) = \begin{cases} \dfrac{1}{2} + \dfrac{1}{2^{(t+1)/2}} & \text{if } 2 \mid n, \\ \\ \dfrac{1}{2} - \dfrac{1}{2^{(t+1)/2}} & \text{if } 2 \nmid n, \end{cases} \hspace{0.4in} \text{and} \hspace{0.4in} \lim_{n \to \infty} \delta_t^o(n) = \begin{cases}  \dfrac{1}{2} - \dfrac{1}{2^{(t+1)/2}} & \text{if } 2 \mid n, \\ \\ \dfrac{1}{2} + \dfrac{1}{2^{(t+1)/2}} & \text{if } 2 \nmid n. \end{cases}$$
\end{theorem}

We also study the sign pattern of $p^e_t(n) - p^o_t(n)$, for $n \rightarrow \infty$, which determines when $p_t^e(n) > p_t^o(n)$ and $p_t^o(n) > p_t^e(n)$.

\begin{theorem}\label{distribution property}
	For $t>1$ a fixed positive integer, write $t = 2^s\ell$ for integers $s, \ell \geq 0$ such that $\ell$ odd. Then for sufficiently large $n$, the sign of $p^e_t(n) - p^o_t(n)$ is periodic with period $2^{s+1}$. In particular, when $t$ is odd the sign of $p^e_t(n) -p^o_t(n)$ is alternating for sufficiently large $n$.
\end{theorem}

\begin{example*}
For $t = 6$ and sufficiently large $n$, we have that $p^e_t(n) - p^o_t(n) > 0$ if and only if $n \equiv 0, 1 \pmod{4}$.
\end{example*}

This paper is organized as follows. In Section 2, we state an exact formula for the difference $p^e_t(n) - p_t^o(n)$ (see Theorem \hyperref[Exact Formula]{\ref{Exact Formula}}), and prove the exact formula by applying the circle method of Rademacher to the generating function of $p^e_t(n) - p_t^o(n)$. In Section 3, we use this exact formula to analyze the limiting behavior of the distributions of $\delta_t^e(n)$ and $\delta_t^o(n)$. Although the consequences regarding distribution properties only require an asymptotic formula for $p_t^e(n) - p_t^o(n)$, for completeness we provide the exact formula. In Section 4, we conclude with remarks on the source of these surprising distribution properties as well as possible implications and generalizations of these results.

\section{Exact Formulas}

\subsection{Auxiliary Partition Functions}

Since $p^e_t(n) + p^o_t(n) = p(n)$, $A_t(n)$ can serve a useful auxiliary role in our study. The utility of the function $A_t(n)$ comes from the generating function 

\begin{equation} \label{A_t(n) Generating Function}
    G_t(x): = \sum\limits_{n \geq 0} A_t(n) x^n = \prod\limits_{k \geq 1} \dfrac{(1 - x^{4tk})^t(1 - x^{tk})^{2t}}{(1 - x^{2tk})^{3t}(1-x^k)},
\end{equation}

\noindent proven in Corollary 5.2 of \cite{Han10}. Using the generating function (\ref{A_t(n) Generating Function}), we prove the following exact formula for $A_t(n)$, given as a Rademacher-type infinite series expansion.

\begin{theorem} \label{Exact Formula}
If $n, t$ are positive integers with $t > 1$, then
\begin{alignat*}{2}
A_t(n) = \dfrac{2^{t/2}}{(24n-1)^{3/4}} \sum_{\substack{k \geq 1 \\ \gcd(k,2t) = 1}} \dfrac{\pi}{k} \sum_{\substack{0 \leq h < k \\ \gcd(h,k) = 1}} e^{\frac{-2\pi i n h}{k}} w(t,h,k) \sum_{m = 0}^{U_{t,k}} e^{\frac{2\pi i (4t)^* H m}{k}} c_1(t,h,k;m) &\\ \cdot \left( \dfrac{t - 24m}{t} \right)^{3/4} I_{\frac 32} \left( \dfrac{\pi}{12k} \sqrt{\dfrac{(t - 24m)(24n-1)}{t}} \right)& \\
+ \dfrac{2^{t/2}}{(24n-1)^{3/4}} \sum_{\substack{k \geq 1 \\ 2 || k_0}} \dfrac{2\pi}{k} \sum_{\substack{0 \leq h < k \\ \gcd(h,k) = 1}} e^{\frac{-2\pi i n h}{k}} w(t,h,k) \sum_{m = 0}^{U_{t,k}} e^{\frac{2\pi i (2^\dagger t_0^*) H m}{k}} c_2(t,h,k;m) &\\ \cdot \left( \dfrac{t_0(1 + 3\gcd(k,t)^2) - 12m}{t_0} \right)^{3/4} I_{\frac 32} \left( \dfrac{\pi}{6k} \sqrt{\dfrac{t_0(1+3\gcd(k,t)^2) - 12m)(24n-1)}{t_0}} \right) &\\
+ \dfrac{1}{(24n-1)^{3/4}} \sum_{\substack{k \geq 1 \\ 4 | k_0}} \dfrac{2\pi}{k} \sum_{\substack{0 \leq h < k \\ \gcd(h,k) = 1}} e^{\frac{-2\pi i n h}{k}} w(t,h,k) \sum_{m = 0}^{U_{t,k}} e^{\frac{2\pi i (t_0^* H) m}{k}} c_3(t,h,k;m) &\\ \cdot \left(\dfrac{t_0 - 24m}{t_0} \right)^{3/4} I_{\frac 32} \left( \dfrac{\pi}{6k} \sqrt{\dfrac{(t_0 - 24m)(24n-1)}{t_0}} \right),&
\end{alignat*}
where $k_0 := \dfrac{k}{\gcd(k,t)}$, $t_0 := \dfrac{t}{\gcd(k,t)}$, $H$ satisfies $hH \equiv -1 \pmod{k}$, $h^*$ (resp. $h^\dagger$) denotes the inverse of $h$ modulo $k_0$ (resp. $k_0/2$), $U_{t,k}$ is defined by
\begin{align*}
U_{t,k} := \begin{cases}
\left\lfloor \dfrac{t}{24} \right\rfloor & \textnormal{if } 2 \centernot | k_0, \\[0.2in]
\left\lfloor \dfrac{t_0(1 + 3\gcd(k,t)^2)}{12} \right\rfloor & \textnormal{if } 2 || k_0, \\[0.2in]
\left\lfloor \dfrac{t_0}{24} \right\rfloor & \textnormal{if } 4 | k_0,
\end{cases}
\end{align*}
$w(t,h,k)$ is defined by \eqref{w-def}, $c_j(t,h,k;m)$ are defined by \eqref{c-def}, and $I_{\frac 32}(z)$ is the classical modified $I$-Bessel function.
\end{theorem}

\begin{example*}
We illustrate Theorem \ref{Exact Formula} using the numbers $A_t(d;n)$,
which denote partial sums for $A_t(n)$ over $1 \leq k \leq d$. Theorem \ref{Exact Formula} is therefore the statement that $\lim\limits_{d \rightarrow \infty} A_t(d;n) = A_t(n)$. We offer some examples in the table below.

\begin{table}[h]
\small
	\begin{tabular}{|c|c|c|c|c|c|c|c|} \hline
		\diagbox[width=\dimexpr \textwidth/6, height=1cm]{\quad\quad $n$}{$d$ \quad} & $10$  & $100$ &$1000$&$\cdots$ & $\infty$ \\ \hline
		$50$ &  $\approx 114580.084$& $\approx 114579.996$ & $\approx 114580.000$  & $\cdots$ & $114580$\\ \hline
		$100$  & $\approx 81486201.594$ & $\approx 81486198.001$ & $\approx 81486198.000$ & $\cdots$ & $81486198$\\ \hline
	\end{tabular}
	\vspace{0.5ex}
	\caption{Values of $A_3(d;n)$}
\end{table}

\end{example*}
\vspace{-5ex}
This exact formula gives the following corollary.

\begin{corollary}\label{Dominating term of A_t}
	For $t>1$ a fixed positive integer, write $t = 2^s\ell$ for integers $s, \ell \geq 0$ such that $\ell$ is odd. Then as $n \rightarrow \infty$ we have
	\begin{equation*} \label{Dominant Term}
	A_t(n) \sim \displaystyle\dfrac{\pi}{2^{s+\frac{t}{2}}}\bigg(\dfrac{1+ 3\cdot 4^s}{24n -1}\bigg)^\frac{3}{4}I_{\frac{3}{2}}\bigg(\dfrac{\pi\sqrt{(1 + 3\cdot 4^s)(24n -1)}}{6\cdot 2^{s+1}}\bigg)\sum_{\substack{0< h < 2^{s+1}\\h \text{ odd}}}w_2(t,h,2^{s+1})e^{-\frac{\pi i nh}{2^s}}.
	\end{equation*}
	
	In particular, when $t$ is odd we have 
	\begin{equation*} \label{Dominant Term when t odd}
	A_t(n) \sim \displaystyle (-1)^n \dfrac{\pi\cdot 2^{(3-t)/2}}{(24n-1)^{3/4}}I_{\frac{3}{2}}\bigg(\dfrac{\pi\sqrt{24n -1}}{6}\bigg).
	\end{equation*}
\end{corollary}

\begin{proof}
	For $z \in \mathbb{R}^+$, it is known that $I_{\frac 32}(z) \sim \dfrac{e^z}{\sqrt{2\pi z}}$. From this asymptotic relation, we can derive a condition for isolating the dominant term in Theorem \ref{Exact Formula}. In particular, let $\{ g_i(n) \}_{i=0}^\infty$ be a countable collection of functions such that $\lim\limits_{n \to \infty} \dfrac{g_0(n)}{g_i(n)} > 1$ for all $i \not = 0$ and let $\{ a_i(n) \}_{i=0}^\infty$ be complex numbers each of which grow at most polynomially in $n$. Then if the series $\sum\limits_{i=0}^\infty a_i(n) I_{3/2}(g_i(n))$ converges and $a_0(n)$ does not vanish, we have
	$$\sum\limits_{i=0}^\infty a_i(n) I_{\frac 32}(g_i(n)) \sim a_0(n) I_{3/2}(g_0(n))$$
	as $n \to \infty$. This reduces the proof to an analysis of the analogs of $g_i(n)$ and $a_0(n)$ in Theorem \ref{Exact Formula}.
	
	Each of the arguments of $I_{\frac 32}(z)$ is maximized when $m = 0$, so we are left with the task of finding the largest possible value of of coefficients on $\sqrt{24n-1}$, which are given by $\dfrac{\pi}{12k}$ if $2 \centernot | k_0$, $\dfrac{\pi}{6k} \sqrt{1 + 3\gcd(k,t)^2}$ if $2 || k_0$, and $\dfrac{\pi}{6k}$ if $4 | k_0$. Among these, it is clear that the case where $2 || k_0$ is the largest. Now, this expression can be rewritten as
	$$\dfrac{\pi}{6k} \sqrt{1 + 3\gcd(k,t)^2} = \dfrac{\pi}{6k} \sqrt{1 + \dfrac{3}{k_0^2} k^2}.$$
	When $k_0$ is held fixed, since $k_0 \geq 2$ this expression is strictly decreasing in $k$, and therefore the optimal choice of $k$ must be of the form $k = 2^s k_0$. It is also clear that $k_0 = 2$ is optimal, and therefore $k = 2^{s+1}$ has the dominant $I$-Bessel function. Since by Lemma \ref{Kloosterman sum nonzero} the associated Kloosterman sum does not vanish, this completes the proof.
\end{proof}

\subsection{Proof of Theorem \ref{Exact Formula}}

\subsubsection{The Circle Method}
The approach that will be utilized in the proof of Theorem \ref{Exact Formula} is commonly referred to as the ``circle method". Initially developed by Hardy and Ramanujan and refined by Rademacher, the circle method has been employed with great success for the past century in additive number theory. The crowning achievement of the circle method lies in producing an exact formula for the partition function $p(n)$, and it has been utilized to produce similar exact formulas for variants of the partition function. A helpful and instructive sketch of the application of Rademacher's circle method to the partition function $p(n)$ is given in Chapter 5 of \cite{GTM41}. Here, we will provide a summary of the circle method, in order to clarify the key steps and the general flow of the argument.

The function $A_t(n)$ has as its generating function $G_t(x)$. Our objective is to use $G_t(x)$ to produce an exact formula for $A_t(n)$. Consider the Laurent expansion of $G_t(x)/x^{n+1}$ in the punctured unit disk. This function has a pole at $x = 0$ with residue $p(n)$ and no other poles. Therefore, by Cauchy's residue theorem we have
\begin{equation}
A_t(n) = \dfrac{1}{2\pi i} \int_C \dfrac{G_t(x)}{x^{n+1}} dx,
\end{equation}
where $C$ is any simple closed curve in the unit disk that contains the origin in its interior. The task of the circle method is to choose a curve $C$ that allows us to evaluate this integral, and this is achieved by choosing $C$ to lie near the singularities of $G_t(x)$, which are the roots of unity. For every positive integer $N$ and every pair of coprime non-negative integers $0 \leq h < k \leq N$, we choose a special contour $C$ in the complex upper half-plane and divide this contour into arcs $C_{h,k}$ near the roots of unity $e^{2\pi i h/k}$. Integration along $C$ can then be expressed as a finite sum of integrals along the arcs $C_{h,k}$, and elementary functions $\psi_{h,k}$ are chosen with behavior similar to $G_t(x)$ near the singularity $e^{2\pi ih/k}$. The functions $\psi_{h,k}$ are found by using properties of $G_t(x)$ deduced from the functional equation of the Dedekind eta function $\eta(\tau) := e^{\pi i \tau / 12} \prod\limits_{n \geq 1} (1 - e^{2\pi i n \tau})$ and the relation between $G_t(x)$ and $\eta(\tau)$ given by
\begin{equation}\label{G_t in eta}
G_t(e^{2\pi i \tau}) = \dfrac{\eta(t\tau)^{2t} \eta(4t\tau)^t}{\eta(\tau)\eta(2t\tau)^{3t}}.
\end{equation}
The error created by replacing $G_t(x)$ by $\psi_{h,k}(x)$ can be estimated, and the integrals of the $\psi_{h,k}$ along $C_{h,k}$ evaluated. This procedure produces estimates that can be used to formulate a convergent series for $A_t(n)$. Our implementation of the circle method will follow along these same lines.

\subsubsection{Notation}

To preface the proof of Theorem \ref{Exact Formula}, we summarize notation which will be used prominently throughout the rest of the paper. The values of $t, n, h$, and $k$ are always non-negative integers. Additionally, we assume $t > 1$ and that $h, k$ satisfy $0 \leq h < k$ and $\gcd(h,k) = 1$. Frequently, it is necessary to remove common factors between $k$ and $t$, and so we define $k_0 := \dfrac{k}{\gcd(k,t)}$ and $t_0 := \dfrac{t}{\gcd(k,t)}$. We will also make use of multiplicative inverses to a variety of moduli, and use distinct notations to distinguish these. We will always denote by $H$ an integer satisfying $hH \equiv -1 \pmod{k}$, and $h^*, h^\dagger$ will denote inverses of $h$ modulo $k_0$ and $k_0/2$ respectively. The complex numbers $x$ and $z$ are related by $x = \exp\left( \dfrac{2\pi i}{k} \left( h + iz \right) \right)$. Note that although $x$ depends on $h$ and $k$, the dependence is suppressed since these values will be clear in context. The notation $x'$ will always be used to denote a modular transformation of the variable $x$. The modular transformations also make use of the Dedekind sum $s(u,v)$, which for any integers $u,v$ is given by
\begin{align*}
s(u,v) := \sum_{m = 1}^v \left(\left( \dfrac{m}{v} \right)\right) \left(\left( \dfrac{um}{v} \right)\right)
\end{align*}
where
\begin{align*}
((m)) := \begin{cases}
m - \lfloor m \rfloor - \dfrac{1}{2} & m \not \in \mathbb{Z}, \\
0 & m \in \mathbb{Z}.
\end{cases}
\end{align*}
These Dedekind sums will always arise in the context of certain roots of unity $e^{\pi i s(u,v)}$, and so it is convenient to adopt the notation $\omega_{u,v} := e^{\pi i s(u,v)}$.

\subsubsection{Transformation formula for $G_t(x)$}
We first recall the transformation formula for the generating function of $p(n)$ (see, for example \cite{Hagis71} or P.96 in \cite{GTM41}).

\begin{theorem}\label{transformation formula for F}
	Let $k,t$ be positive integers with $t > 1$ and $0 \leq h < k$ an integer coprime to $k$ and $H$ an integer satisfying $hH \equiv -1 \pmod{k}$. Let $z$ be a complex number with $\textnormal{Re}(z) > 0$ and let $x, x'$ be defined by $x = \exp\bigg(\dfrac{2\pi i}{k} (h + iz) \bigg)$ and $x' = \exp\bigg(\dfrac{2\pi i}{k} \left(  H + \dfrac{i}{z} \right)\bigg)$. If $F(x)$ is defined by $F(x) := \prod\limits_{m=1}^\infty \dfrac{1}{1 - x^m}$, then \begin{equation*}
	F(x)  = \sqrt{z} \cdot \omega_{h,k} \exp \bigg(\dfrac{\pi(z^{-1} -z)}{12k}\bigg)F(x').
	\end{equation*}
\end{theorem}

By (\ref{G_t in eta}), $G_t(x)$ can be expressed in terms of $F(x)$:
$$G_t(x)  = \dfrac{F(x)\big[F\big(x^{2t}\big)\big]^{3t}}{\big[F\big(x^{t}\big)\big]^{2t}\big[F\big(x^{4t}\big)\big]^{t}}.$$
We can therefore apply Theorem \hyperref[transformation formula for F]{\ref{transformation formula for F}} to find a similar transformation formula for $G_t(x)$.

\begin{lemma} \label{Transformation Laws}
Define $x_1 := x^t$, $x_2 := x^{2t}$, and $x_3 := x^{4t}$. Then the following transformation formulas for $F(x_j)$ hold. \\

\noindent (a) When $k_0$ is odd, for $1 \leq j \leq 3$ we have $$F(x_j) = \sqrt{2^{j-1} t_0 z} \cdot \omega_{2^{j-1} t_0 h, k_0} \exp \bigg[ \dfrac{\pi}{12k_0} \bigg( \dfrac{1}{2^{j-1} t_0 z} - 2^{j-1} t_0 z \bigg) \bigg] F(x_j')$$
hold, where $x_j' = \exp\bigg[\dfrac{2\pi i}{k_0} \bigg( (2^{j-1} t_0)^* H + \dfrac{i}{2^{j-1} t_0 z}\bigg)\bigg]$. \\

\noindent (b) Suppose $k_0 \equiv 2 \pmod{4}$. Then we have the transformation formulas
$$F(x_1) = \sqrt{t_0 z} \cdot \omega_{t_0 h, k_0} \exp\Bigg[ \dfrac{\pi}{12k_0} \bigg( \dfrac{1}{t_0 z} - t_0 z \bigg) \Bigg] F(x_1')$$
where $x_1' = \exp\bigg[\dfrac{2\pi i}{k_0} \bigg( t_0^* H + \dfrac{i}{t_0 z}\bigg)\bigg]$, and for $j = 2, 3$ we have
$$F(x_j) = \sqrt{2^{j-2} t_0 z} \cdot \omega_{2^{j-2} t_0 h, k_0/2} \exp\Bigg[ \dfrac{\pi}{6k_0} \bigg( \dfrac{1}{2^{j-2} t_0 z} - 2^{j-2} t_0 z \bigg) \Bigg] F(x_j'),$$
where $x_2' = \exp\bigg[\dfrac{2\pi i}{k_0 / 2} \bigg( t_0^* H + \dfrac{i}{t_0 z}\bigg)\bigg],$ and $x_3' = \exp\bigg[\dfrac{2\pi i}{k_0 / 2} \bigg( 2^\dagger t_0^* H + \dfrac{i}{t_0 z}\bigg)\bigg].$ \\

\noindent (c) Suppose $4 | k_0$. Then we have the transformation formulas
$$F(x_j) = \sqrt{t_0 z} \cdot \omega_{t_0 h, k_0 / 2^{j-1}} \exp\Bigg[ \dfrac{2^{3-j} \pi}{12k_0} \bigg( \dfrac{1}{t_0 z} - t_0 z \bigg) \Bigg] F(x_j'),$$
where $x_j' = \exp\Bigg[ \dfrac{2\pi i}{k_0} \bigg( 2^{j-1} t_0^* H + \dfrac{i}{2^{1-j}t_0 z} \bigg) \Bigg]$.
\end{lemma}

\begin{proof}
We first prove the case $j=1$ of (a) By definition, $x_1 = \exp \bigg( \dfrac{2\pi i}{k_0} (t_0 h + i t_0 z) \bigg)$. Since $\gcd(t_0 h, k_0) = 1$ and $t_0 h (t_0^* H) \equiv -1 \pmod{k_0}$, applying Theorem \hyperref[transformation formula for F]{\ref{transformation formula for F}} gives the result by the substitutions $h \mapsto t_0 h$, $k \mapsto k_0$, and $z \mapsto t_0 z$. Every other case of the result follows by rearranging terms in $x_j$ in a manner such that the terms playing the roles of $h$ and $k$ in Theorem \ref{transformation formula for F} are coprime.
\end{proof}

Using these identities for each case, we obtain the transformation law
\begin{align*}
G_t(x) = \begin{cases}
2^{t/2} \sqrt{z} \exp \bigg[ \dfrac{\pi}{12k} \bigg( \dfrac{4 - 3\gcd(k,t)^2}{4z} - z \bigg) \bigg] w_1(t,h,k) \dfrac{F(x') [F(x_2')]^{3t}}{[F(x_1')]^{2t}[F(x_3')]^t}& 2 \centernot | k_0, \\[+0.2in] 
 2^{t/2} \sqrt{z} \exp\bigg[ \dfrac{\pi}{12k}\bigg(\dfrac{1+3\gcd(k,t)^2}{z} - z\bigg)\bigg]w_2(t,h,k)\dfrac{F(x')\big[F(x_2')\big]^{3t}}{\big[F(x_1')\big]^{2t}\big[F(x_3')\big]^{t}} & 2 || k_0, \\[+0.2in] 
\sqrt{z}\exp\bigg[\dfrac{\pi}{12k}\left(\dfrac{1}{z} - z\right)\bigg]w_3(t,h,k)\dfrac{F(x')\big[F(x_2')\big]^{3t}}{\big[F(x_1')\big]^{2t}\big[F(x_3')\big]^{t}}& 4 | k_0,
\end{cases}
\end{align*}
\vspace{0.6ex}
where $w_1(t,h,k) := \omega_{h,k} \omega^{3t}_{2t_0 h, k_0} \omega^{-2t}_{t_0 h, k_0} \omega^{-t}_{4 t_0 h, k_0}$, $w_2(t,h,k) := \omega_{h,k} \omega^{3t}_{t_0 h, k_0/2} \omega^{-2t}_{t_0 h, k_0} \omega^{-t}_{2 t_0 h, k_0/2}$, and $w_3(t,h,k) := \omega_{h,k} \omega^{3t}_{t_0 h, k_0/4} \omega^{-2t}_{t_0 h, k_0} \omega^{-t}_{t_0 h, k_0/4}$. From the definition of $s(h,k)$ we can see that $s(dh, dk) = s(h,k)$ for every integer $d$, and therefore $\omega_{t_0 h, k_0/2} = \omega_{2t_0 h, k_0}$ when $2 | k_0$ and $\omega_{t_0 h, k_0/4} = \omega_{2t_0 h, k_0/2} = \omega_{4t_0 h, k_0}$ when $4 | k_0$. Therefore $w_j(t,h,k) = w(t,h,k)$ for all $j$, where
\begin{align} \label{w-def}
w(t,h,k) := \dfrac{\omega_{h,k} \omega^{3t}_{2t_0 h, k_0} }{\omega^{2t}_{t_0 h, k_0} \omega^{t}_{4 t_0 h, k_0}}.
\end{align}
Therefore, the transformation law for $G_t(x)$ can be rewritten as
\begin{align} \label{Transformation Law for G_t}
G_t(x) = \begin{cases}
2^{t/2} \sqrt{z} \exp \bigg[ \dfrac{\pi}{12k} \bigg( \dfrac{4 - 3\gcd(k,t)^2}{4z} - z \bigg) \bigg] w(t,h,k) J_{t,h,k}(x') & 2 \centernot | k_0, \\ 
 2^{t/2} \sqrt{z} \exp\bigg[ \dfrac{\pi}{12k}\bigg(\dfrac{1+3\gcd(k,t)^2}{z} - z\bigg)\bigg]w(t,h,k) J_{t,h,k}(x') & 2 || k_0, \\ 
\sqrt{z}\exp\bigg[\dfrac{\pi}{12k}\left(\dfrac{1}{z} - z\right)\bigg] w(t,h,k) J_{t,h,k}(x') & 4 | k_0,
\end{cases}
\end{align}
where for shorthand we define $J_{t,h,k}(x') := \dfrac{F(x')\big[F(x_2')\big]^{3t}}{\big[F(x_1')\big]^{2t}\big[F(x_3')\big]^{t}}$.

\subsubsection{A convergent series for $G_t(x)$}

\phantom{x}
We follow closely to the notations and proofs in Chapter 5 of \cite{GTM41}. Let notation be as before, and let $N$ be any positive integer. Recall that
\begin{align*}
G_t(x):= \sum\limits_{\lambda \in \mathcal P}x^{\vert \lambda \vert}(-1)^{H_t(\lambda)} =\sum\limits_{n \geq 0}\sum\limits_{\lambda \vdash n}(-1)^{H_t(\lambda)}x^n = \sum\limits_{n \geq 0}A_t(n)x^n.
\end{align*}
By Cauchy's residue theorem, we have
\begin{align*}
A_t(n) = \frac{1}{2\pi i}\int_C \frac{G_t(x)}{x^{n + 1}}\,dx,
\end{align*}
where $C$ is any positively oriented simple closed curve in a unit disk that contains the origin in its interior. In our implementation of the circle method, we set $C = C_N$ for $C_N$ centered at zero with radius $e^{-2\pi N^{-2}}$.  Using the transformations $x = e^{2\pi i \tau}$ and $z = -ik^2\bigg(\tau -\dfrac{h}{k}\bigg)$ in succession, the circle $C_N$ is mapped onto the circle $\mathcal{K}$ with center $\dfrac{1}{2}$ and radius $\dfrac{1}{2}$. In the rest of this proof, $\mathcal{K}$ will denote this same circle. If we breakdown $C_N$ into Farey arcs, then this change of variables gives the formula
\begin{eqnarray*}
	A_t(n) &=& \displaystyle\sum\limits_{k=1}^N \Bigg[\dfrac{i}{k^2} \sum\limits_{\substack{0\leq h < k\\(h,k) =1}}e^{-\frac{2\pi i nh}{k}}\int_{z_1(h,k)}^{z_2(h,k)} G_t\left( e^{\frac{2\pi i}{k}\left( h + \frac{iz}{k} \right)} \right) e^{\frac{2\pi nz}{k^2}}\,dz\Bigg],
\end{eqnarray*}
where the integral runs along the arc of $\mathcal{K}$ between the points $z_1(h,k)$ and $z_2(h,k)$ defined by
$$z_1(h,k) = \dfrac{k^2}{k^2 + k_1^2} + i \dfrac{kk_1}{k^2 + k_1^2} \ \ \text{ and } \ \ z_2(h,k) = \dfrac{k^2}{k^2 + k_2^2} - i \dfrac{kk_2}{k^2 + k_2^2},$$
where $k_1, k, k_2$ are the denominators of consecutive terms of the Farey series of order $N$. Computing $A_t(n)$ therefore reduces to computing the integrals
\begin{align*}
I(t,h,k,n) := \int_{z_1(h,k)}^{z_2(h,k)} G_t\left( e^{\frac{2\pi i}{k}\left( h + \frac{iz}{k} \right)} \right) e^{\frac{2\pi n z}{k^2}} dz.
\end{align*}
The first step to evaluating these integrals is an application of the transformation law for $G_t(x)$. Because of the formulation of \eqref{Transformation Law for G_t}, the exact formula is naturally broken into three sums. One of these is given by
$$\sum_{\substack{k \geq 1 \\ k_0 \textnormal{ odd}}} \sum\limits_{\substack{0 \leq h < k \\ \gcd(h,k) = 1}} e^{-2\pi i n h / k} I(t,h,k,n)$$
and the other two are defined similarly with the modification that $k_0$ odd is replaced by either $2 || k_0$ or $4 | k_0$. Because of this natural breakdown, the evaluation of $I(t,h,k,n)$ also naturally breaks into three cases.

In order to estimate the integrals $I(t,h,k,n)$, we will use a series expansion for the factor $J_{t,h,k}(x') = \dfrac{F(x') [F(x_2')]^{3t}}{[F(x_1')]^{2t}[F(x_3')]^t}$ in the modular transformation law for $G_t(x)$. The variable we will use for this series expansion will depend on the value of $k_0$. In particular, define $y_j$ for $1 \leq j \leq 3$ by
$$y_1 := e^{\frac{2\pi i}{k} \left( 4^* t_0^* H + \frac{i}{4 t_0 z} \right)}, \ \ \ y_2 := e^{\frac{2\pi i}{k} \left( 2^\dagger t_0^* H + \frac{i}{2 t_0 z} \right)}, \ \ \ y_3 := e^{\frac{2\pi i}{k} \left( t_0^* H + \frac{i}{t_0 z} \right)}.$$
The utility of using $y_j$ is that it relates nicely to the variables $x'$, $x_1'$, $x_2'$, and $x_3'$ appearing in $J_{t,h,k}(x')$. From definitions, it follows that
\vspace{-1ex}
\begin{center}
$$x_1' = y_1^{4\gcd(k,t)} = - y_2^{2\gcd(k,t)} = y_3^{\gcd(k,t)},$$
$$x_2' = y_1^{2\gcd(k,t)} = y_2^{4\gcd(k,t)} = y_3^{2\gcd(k,t)},$$
$$x_3' = y_1^{\gcd(k,t)} = y_2^{2\gcd(k,t)} = y_3^{4\gcd(k,t)},$$
\end{center}
and
$$x' = y_1^{4t_0} e^{\frac{-2\pi i \left( 4t_0(4t_0)^* - 1 \right) H}{k}} = y_2^{2t_0} e^{\frac{-2\pi i \left( 2 2^\dagger t_0 t_0^* - 1 \right)H}{k}} = y_3^{t_0} e^{\frac{-2\pi i \left( t_0 t_0^* - 1 \right) H}{k}}.$$ 
Therefore, we have three series expansions for $J_{t,h,k}(x')$ given by
\begin{align} \label{c-def}
J_{t,h,k}(x') =: \sum_{m \geq 0} c_j(t,h,k; m) y_j^m
\end{align}
for $1 \leq j \leq 3$. These series expansions, along with the transformation laws for $G_t(x)$, are used to aid in the evaluation of the integrals $I(t,h,k,n)$.

\subsubsection{Estimating $I(t,h,k,n)$}

The process of evaluating $I(t,h,k,n)$ breaks into three cases based on the value of $k_0$. Since the proofs in every case run along similar lines, we need only write out details in the case where $k_0$ is odd and to comment on which aspects of the proof need to be altered for the other two cases. When $k_0$ is odd, we use the series expansion for $J_{t,h,k}(x')$ in $y_1$. Applying the substitution $z \mapsto \dfrac{z}{k}$ in \eqref{Transformation Law for G_t}, we have
$$I(t,h,k,n) = \dfrac{2^{t/2} w(t,h,k)}{\sqrt{k}} \int_{z_1(h,k)}^{z_2(h,k)} \sum_{m \geq 0} e^{\frac{2\pi i (4t_0)^* H m}{k}} c_1(t,h,k;m) f_{k,t,m}(z) e^{\frac{2\pi n z}{k^2}} \, dz$$
where
$$f_{k,t,m}(z) := \sqrt{z} \exp \bigg[ \dfrac{\pi}{12} \bigg( \dfrac{4 - 3\gcd(k,t)^2}{4z} - \dfrac{6 m}{t_0 z} - \dfrac{z}{k^2} \bigg) \bigg].$$
From the theory of Farey arcs (see Theorem 5.9 of \cite{GTM41}) we know that the path of integration has length less than $2\sqrt{2} k N^{-1}$ and that for any $z$ on the path of integration, $|z| < \sqrt{2} k N^{-1}$. Furthermore, any $z \in \mathcal{K} \backslash \{0\}$ satisfies $0 < \textnormal{Re}(z) \leq 1$ and $\textnormal{Re}(1/z) = 1$. From these facts, we can see that $m > M_{t,k} := \left\lfloor \dfrac{t_0 \left( 4 - 3\gcd(k,t)^2 \right)}{24} \right\rfloor$ if and only if
$$\left| e^{\frac{\pi}{12} \left( \frac{4 - 3\gcd(k,t)^2}{4z} - \frac{6 m}{t_0 z} - \frac{z}{k^2} \right)} \right| < 1,$$
and that therefore
$$\int_{z_1(h,k)}^{z_2(h,k)} \sum_{m > M_{t,k}} e^{\frac{2\pi i (4t_0)^* H m}{k}} c_1(t,h,k;m) f_{k,t,m}(z) e^{\frac{2\pi n z}{k^2}} \, dz = O\left( k^{3/2} N^{-3/2} \right).$$
Applying this estimate to $I(t,h,k,n)$, it follows that
$$I(t,h,k,n) = \dfrac{2^{t/2} w(t,h,k)}{\sqrt{k}} \int_{z_1(h,k)}^{z_2(h,k)} \sum_{m = 0}^{M_{t,k}} e^{\frac{2\pi i (4t_0)^* H m}{k}} c_1(t,h,k;m) f_{k,t,m}(z) e^{\frac{2\pi n z}{k^2}} \, dz + O\left( k^{1/2} N^{-3/2} \right).$$
Similar estimates apply in the other two cases. In particular, extend the definition of $M_{t,k}$ by
\begin{align*}
M_{t,k} := \begin{cases}
\left\lfloor \dfrac{t_0(4 - 3\gcd(k,t)^2)}{24} \right\rfloor & \textnormal{if } 2 \centernot | k_0, \\[+0.2in]
\left\lfloor \dfrac{t_0(1 + 3\gcd(k,t)^2)}{12} \right\rfloor & \textnormal{if } 2 || k_0, \\[+0.2in]
\left\lfloor \dfrac{t_0}{24} \right\rfloor & \textnormal{if } 4 | k_0
\end{cases}
\end{align*}
and in place of $y_1$ use $y_2$ when $2 || k_0$ or $y_3$ when $4 | k_0$. These modifications lead to the following proposition.

\begin{proposition} \label{Reduce to Finite Sum}
Adopt all notation as above. Then if $k_0$ is odd,
\begin{align*}
I(t,h,k,n) =& \dfrac{2^{t/2} w(t,h,k)}{\sqrt{k}} \sum_{m = 0}^{M_{t,k}} e^{\frac{2\pi i (4t_0)^* H m}{k}} c_1(t,h,k;m) \\ &\cdot \int_{z_1(h,k)}^{z_2(h,k)} \sqrt{z} \exp \bigg[ \dfrac{\pi}{12} \left( \dfrac{4 - 3\gcd(k,t)^2}{4z} - \dfrac{6 m}{t_0 z} + \dfrac{(24n-1)z}{k^2} \right) \bigg] \, dz + O\left( k^{1/2} N^{-3/2} \right).
\end{align*}
If $2 || k_0$, then we have
\begin{align*}
I(t,h,k,n) =& \dfrac{2^{t/2} w(t,h,k)}{\sqrt{k}} \sum_{m = 0}^{M_{t,k}} e^{\frac{2\pi i (2^\dagger t_0^*) H m}{k}} c_2(t,h,k;m) \\ &\cdot \int_{z_1(h,k)}^{z_2(h,k)} \sqrt{z} \exp\bigg[ \dfrac{\pi}{12} \left( \dfrac{1+3\gcd(k,t)^2}{z} - \frac{12m}{t_0 z} + \dfrac{(24n-1)z}{k^2} \right) \bigg] \, dz + O\left( k^{1/2} N^{-3/2} \right).
\end{align*}
If $4 | k_0$, then we have
\begin{align*}
I(t,h,k,n) =& \dfrac{w(t,h,k)}{\sqrt{k}} \sum_{m = 0}^{M_{t,k}} e^{\frac{2\pi i (t_0^* H) m}{k}} c_3(t,h,k;m) \\ &\cdot \int_{z_1(h,k)}^{z_2(h,k)} \sqrt{z} \exp\left[ \dfrac{\pi}{12} \left( \dfrac{1}{z} - \dfrac{24 m}{t_0 z} + \dfrac{(24n-1)z}{k^2} \right) \right] \, dz + O\left( k^{1/2} N^{-3/2} \right).
\end{align*}
\end{proposition}

In light of Proposition \ref{Reduce to Finite Sum}, the problem of evaluating $I(t,h,k,n)$ is reduced to evaluating integrals of the form
$$\int_{z_1(h,k)}^{z_2(h,k)} \sqrt{z} \exp\left[ \dfrac{\pi}{12} \left( \dfrac{A - Bm}{z} + \dfrac{(24n-1z)}{k^2} \right) \right] \, dz$$
for certain constants $A,B$. This evaluation has two main steps. Firstly, we show that extending the path of integration to the whole circle $\mathcal{K}$ introduces only a small error term. Secondly, we show how the integral along $\mathcal{K}$ is expressible by familiar functions from analysis. These steps are carried out together in the following proposition.

\begin{proposition} \label{General I-Bessel}
Fix an integer $t > 1$, and let $A,B$ be constants independent of $z$ for which $A = O_{k}(1)$ as $N \to \infty$ and $B > 0$. Then we have
\begin{align*}
\dfrac{1}{2\pi i} \int_{z_1(h,k)}^{z_2(h,k)} \sqrt{z} e^{\frac{\pi}{12} \left( \frac{A - Bm}{z} + \frac{(24n-1)z}{k^2} \right)} \, dz =& \dfrac{k^{3/2} (A - Bm)^{3/4}}{(24n-1)^{3/4}} I_{\frac 32} \left( \dfrac{\pi}{6k} \sqrt{(A - Bm)(24n-1)} \right) \\ &+ O\left( k^{3/2} N^{-3/2} \right).
\end{align*}
\end{proposition}

\begin{proof}
For $\mathcal{K}^-$ the negative orientation of the circle $\mathcal{K}$, we can break down integrals over $\mathcal{K}^-$ by
$$\int_{\mathcal{K}^-} = \int_{z_1(h,k)}^{z_2(h,k)} + \int_0^{z_1(h,k)} + \int_{z_2(h,k)}^0.$$
Define the function $f(z)$ by
$$f(z) := \sqrt{z} \exp\left[ \dfrac{\pi}{12} \left( \dfrac{A - Bm}{z} + \dfrac{(24n-1)z}{k^2} \right) \right].$$
Then by the theory of Farey arcs, the arc on $\mathcal{K}^-$ from $0$ to $z_1(h,k)$ has length less than $\pi |z_1(h,k)| < \sqrt{2} \pi k N^{-1}$ and therefore $|z| < \sqrt{2} k N^{-1}$ on the path of integration. Recalling that $\textnormal{Re}(1/z) = 1$ and $0 < \textnormal{Re}(z) \leq 1$ on $\mathcal{K} \backslash \{0\}$,
$$\left| \int_0^{z_1(h,k)} f(z) \, dz \right| \leq \dfrac{2^{3/4} \pi k^{3/2}}{N^{3/2}} \exp\left[ \dfrac{\pi}{12} \left( A + 24n-1 \right) \right] = O\left( k^{3/2} N^{-3/2} \right).$$
A similar estimate holds for integrals from $z_2(h,k)$ to $0$, and therefore we have
$$\int_{z_1(h,k)}^{z_2(h,k)} f(z) \, dz = \int_{\mathcal{K}^-} f(z) \, dz + O\left( k^{3/2} N^{-3/2} \right).$$
It suffices now to evaluate the integral
$$I := \int_{\mathcal{K}^-} \sqrt{z} \exp\left[ \dfrac{\pi}{12} \left( \dfrac{A - Bm}{z} + \dfrac{(24n-1)z}{k^2} \right) \right] \, dz.$$
The substitution $w = z^{-1}$, $dw = - z^{-2} dz$ implies
$$I = - \int_{1-i\infty}^{1+i\infty} w^{-5/2} \exp\left( \dfrac{\pi(A - Bm)}{12} w + \dfrac{\pi(24n-1)}{12k^2} w^{-1} \right) dw.$$
Furthermore, by the substitution $s = c w$ for $c := \dfrac{\pi(A - Bm)}{12}$ we have
$$I = - \left( \dfrac{\pi(A - Bm)}{12} \right)^{3/2} \int_{c-i\infty}^{c + i\infty} s^{-5/2} \exp\left( s + \left( \dfrac{\pi^2(A - Bm)(24n-1)}{144k^2} \right) \dfrac{1}{s} \right) ds.$$
Since the classical modified $I$-Bessel function $I_{\frac 32}(z)$ satisfies the identity 
$$I_{\frac 32}(z) = \dfrac{(z/2)^{3/2}}{2\pi i} \int_{c-i\infty}^{c + i\infty} s^{-5/2} \exp\left( s + \dfrac{z^2}{4s} \right) \, ds,$$
setting $\dfrac{z}{2} = \sqrt{\dfrac{\pi^2(A - Bm)(24n-1)}{144k^2}} = \dfrac{\pi}{12k} \sqrt{(A - Bm)(24n-1)}$ yields
\begin{align*}
I &= \dfrac{2\pi}{i} \cdot \dfrac{k^{3/2} (A - Bm)^{3/4}}{(24n-1)^{3/4}} I_{\frac 32} \left( \dfrac{\pi}{6k} \sqrt{(A - Bm)(24n-1)} \right).
\end{align*}
Combining the estimation and the evaluation of $I$ completes the proof.
\end{proof}

From Proposition \ref{General I-Bessel}, we may complete the proof of the exact formula. The idea is that the error term in the evaluation of $A_t(n)$ introduced by the error in $I(t,h,k,n)$ vanishes as $N \to \infty$, and the resulting series converges.

\subsubsection{Proof of Theorem \ref{Exact Formula}}

We have shown that
\begin{align*}
	A_t(n) = \sum\limits_{k=1}^N \dfrac{i}{k^2} \sum\limits_{\substack{0\leq h < k\\(h,k) =1}} e^{-\frac{2\pi i nh}{k}} I(t,h,k,n).
\end{align*}
By Proposition \ref{Reduce to Finite Sum} and Proposition \ref{General I-Bessel}, we obtain for every pair $h, k$ estimates for $I(t,h,k,n)$ with error term $O(k^{1/2} N^{-3/2})$. These exact formulas yield an estimate for $A_t(n)$ with error term $O(N^{-1/2})$. Therefore, as $N \to \infty$ we may replace $I(t,h,k,n)$ with these estimates and retain equality. That is,
\begin{align*}
A_t(n) &= \sum\limits_{k=1}^\infty \dfrac{i}{k^2} \sum\limits_{\substack{0\leq h < k\\(h,k) =1}} e^{-\frac{2\pi i nh}{k}} I(t,h,k,n).
\end{align*}
This exact formula naturally splits into three sums according to the value of $k_0 \pmod{4}$. When $k_0$ is odd, the formula derived from Propositions \ref{Reduce to Finite Sum} and \ref{General I-Bessel} give the contribution
\begin{align*}
S_1 := 2^{t/2} \sum_{\substack{k \geq 1 \\ k_0 \textnormal{ odd}}} \dfrac{2\pi}{k} \sum_{\substack{0 \leq h < k \\ \gcd(h,k) = 1}} e^{\frac{-2\pi i n h}{k}} w(t,h,k) \sum_{m = 0}^{M_{t,k}}& e^{\frac{2\pi i (4t_0)^* H m}{k}} c_1(t,h,k;m) \dfrac{(A - Bm)^{3/4}}{(24n-1)^{3/4}} \\ &\cdot I_{\frac 32} \left( \dfrac{\pi}{6k} \sqrt{(A - Bm)(24n-1)} \right),
\end{align*}
where $A = 1 - \dfrac{3}{4} \gcd(k,t)^2$, $B = \dfrac{6}{t_0}$, and $M_{t,k} = \left\lfloor \dfrac{t_0 (4 - 3\gcd(k,t)^2)}{24} \right\rfloor$. Noting that the sum is only nonempty when $k$ is odd and $\gcd(k,t) = 1$, in which case $k_0 = k$, $t_0 = t$, $M_{t,k} = \left\lfloor \dfrac{t}{24} \right\rfloor$, $A = 1/4$ and $B = 6/t$ we have
\begin{align*}
S_1 = \dfrac{2^{t/2}}{(24n-1)^{3/4}} \sum_{\substack{k \geq 1 \\ \gcd(k,2t) = 1}} \dfrac{\pi}{k} \sum_{\substack{0 \leq h < k \\ \gcd(h,k) = 1}} e^{\frac{-2\pi i n h}{k}} w(t,h,k) \sum_{m = 0}^{\lfloor \frac{t}{24} \rfloor}& e^{\frac{2\pi i (4t)^* H m}{k}} c_1(t,h,k;m) \left( \dfrac{t - 24m}{t} \right)^{3/4} \\ &\cdot I_{\frac 32} \left( \dfrac{\pi}{12k} \sqrt{\dfrac{(t - 24m)(24n-1)}{t}} \right).
\end{align*}
The sums $S_2$, $S_3$ simplify similarly to
\begin{align*}
S_2 = \dfrac{2^{t/2}}{(24n-1)^{3/4}} \sum_{\substack{k \geq 1 \\ 2 || k_0}} \dfrac{2\pi}{k} \sum_{\substack{0 \leq h < k \\ \gcd(h,k) = 1}} e^{\frac{-2\pi i n h}{k}} w(t,h,k) \sum_{m = 0}^{\lfloor \frac{t_0\alpha_{t,k}}{12} \rfloor}& e^{\frac{2\pi i (2^\dagger t_0^*) H m}{k}} c_2(t,h,k;m) \left( \dfrac{t_0\alpha_{t,k} - 12m}{t_0} \right)^{3/4} \\ &\cdot I_{\frac 32} \left( \dfrac{\pi}{6k} \sqrt{\dfrac{t_0\alpha_{t,k} - 12m)(24n-1)}{t_0}} \right)
\end{align*}
where $\alpha_{t,k} := 1 + 3\gcd(k,t)^2$ and
\begin{align*}
S_3 = \dfrac{1}{(24n-1)^{3/4}} \sum_{\substack{k \geq 1 \\ 4 | k_0}} \dfrac{2\pi}{k} \sum_{\substack{0 \leq h < k \\ \gcd(h,k) = 1}} e^{\frac{-2\pi i n h}{k}} w(t,h,k) \sum_{m = 0}^{\lfloor \frac{t_0}{24} \rfloor}& e^{\frac{2\pi i (t_0^* H) m}{k}} c_3(t,h,k;m) \left(\dfrac{t_0 - 24m}{t_0} \right)^{3/4} \\ &\cdot I_{\frac 32} \left( \dfrac{\pi}{6k} \sqrt{\dfrac{(t_0 - 24m)(24n-1)}{t_0}} \right).
\end{align*}
As $A_t(n) = S_1 + S_2 + S_3$, the proof is complete.

\section{Proof of Theorem \ref{distribution property}}

\subsection{Preliminaries}
We start by proving that the Kloosterman sum $$\displaystyle\sum_{\substack{0\leq h < k\\(h,k) =1}} \exp\bigg[\pi i \bigg(s(h,k) -\dfrac{2nh}{k}\bigg)\bigg]$$ is nonzero when $k$ is a power of $2$. Note that this Kloosterman sum can also  be rewritten as a sum of solutions modulo $24k$ to a quadratic equation as defined in the lemma below. 

\begin{lemma}
Let $S_k(n)$ be the Kloosterman sum defined by
\begin{equation} \label{Kloosterman Sum Definition} 
S_k(n) := \dfrac{1}{2}\sqrt{\dfrac{k}{12}} \sum\limits_{\substack{x \pmod{24k} \\ x^2 \equiv -24n + 1 \pmod{24k}}} \chi_{12}(x) e\bigg( \dfrac{x}{12k} \bigg),
\end{equation}
where $\chi_{12}(x) = { \genfrac(){1pt}{0}{12}{x} }$ is the Kronecker symbol and $e(x) := e^{2\pi ix}$. If $k$ is a power of $2$, then $S_k(n) \not = 0$ for all positive integers $n$.
\end{lemma}

\begin{proof}
Let $n \geq 1$, and let $k = 2^s$ for an integer $s \geq 0$. To show that $S_k(n) \not = 0$, we need only show that the summation given in (\ref{Kloosterman Sum Definition}) is nonzero. To evaluate this sum, consider the condition on $x$ that $x^2 \equiv -24n + 1$ (mod $24k$). Since $-24n + 1 \equiv 1$ (mod 4), $x^2 \equiv -24n+1$ (mod $2^{s+3}$) has exactly 4 incongruent solutions, and so the congruence $x^2 \equiv -24n+1$ (mod $24k$) has exactly 8 incongruent solutions. For any given solution $x$, we can see that all of $12k-x$, $12k+x$, and $24k-x$ are also solutions and are pairwise distinct.

Now, let $x,y$ (mod $24k$) be solutions to $x^2 \equiv -24n+1$ (mod $24k$) such that $y$ is not congruent to any of $x$, $12k-x$, $12k+x$, or $24k-x$, so that the summation in (\ref{Kloosterman Sum Definition}) runs over the set of eight values $\{ \pm x, \pm (12k+x)\} \cup \{ \pm y, \pm (12k+y) \}$. Taking real parts in the summation in (\ref{Kloosterman Sum Definition}) yields the value $4a + 4b$, where $a = \chi_{12}(x)\cos{(\pi x / 6k)}$ and $b = \chi_{12}(y)\cos{(\pi y / 6k)}$. The equivalences known about $x$ and $y$ imply that $\chi_{12}(x), \chi_{12}(y) \not = 0$, and so the proof reduces to demonstrating that $|a| \not = |b|$. If $|a| = |b|$, then $x \equiv y$ (mod $6k$) must hold, so we may fix $y = 6k - x$. Since $x$ is odd, $y^2 = x^2 - 12kx + 36k^2 \equiv -24n+1 + 12k + 36k^2 \pmod{24k},$ and the equivalence modulo $6k$ of $x$ and $y$ implies $12k + 36k^2 \equiv 12k(1+3k) \equiv 0$ (mod $24k$). This requires that $1 + 3k$ be even, which is a contradiction since $k = 2^s$. Therefore, $|a| \not = |b|$, and it then follows that $S_k(n) \not = 0$ for all $n$.
\end{proof}

\begin{lemma}\label{Kloosterman sum nonzero}
For $t>1$ a fixed positive integer, write $t = 2^s\ell$  with integers $s, \ell \geq 0$ such that $\ell$ is odd. Then $$\sum_{\substack{0< h < 2^{s+1} \\ h \text{ odd}}} w(t,h,2^{s+1}) e^{-\frac{\pi i n h}{2^s}} = S_{2^s}(n) \neq 0.$$
\end{lemma}

\begin{proof}
By making use of the fact that $\omega_{dh, dk} = \omega_{h,k}$ for any integer $d$, it follows that $w(t,h,2^{s+1}) = \omega_{h,k}$, and therefore
$$\sum_{\substack{0< h < 2^{s+1} \\ h \text{ odd}}} w(t,h,2^{s+1}) e^{-\frac{\pi i n h}{2^s}} = \sum_{\substack{0 < h < 2^{s+1} \\ h \textnormal{ odd}}} e^{\pi i \left( s(h,k) - n h / 2^s \right)} = S_{2^s}(n),$$
which is non-vanishing by Lemma \ref{Kloosterman sum nonzero}.
\end{proof}

\subsection{Proofs of Theorems \ref{Even and Odd T Behavior} and \ref{distribution property}} 

\phantom{X}

We are now ready to prove the main theorems. 

\begin{proposition}\label{ratios of  A_t(n): p(n)}
	Let $n$ be positive integers. Then for $t$ fixed, as $n \to \infty$ we have
	$$\dfrac{A_t(n)}{p(n)} \sim \begin{cases} (-1)^n / 2^{(t-1)/2} & \text{ if } 2 \nmid t, \\ 0 & \text{ if } 2 \mid t. \end{cases}$$
	Furthermore, $\dfrac{A_t(n)}{p(n)} \sim 0$ as $n,t \to \infty$.
\end{proposition}

\begin{proof}
    Recall that $p(n)$ satisfies $p(n) \sim \dfrac{2\pi}{(24n -1)^{3/4}}I_{\frac{3}{2}}\bigg(\dfrac{\pi\sqrt{24n -1}}{6}\bigg)$ as $n \to \infty$. Then by Corollary \hyperref[Dominating term of A_t]{\ref{Dominating term of A_t}}, as $n \rightarrow \infty$ we have  $$\dfrac{A_t(n)}{p(n)} \sim \dfrac{(1+ 3\cdot 4^s)^{3/4}}{2^{s+1+ \frac{t}{2}}}\cdot\dfrac{I_{\frac{3}{2}}\bigg(\dfrac{\pi}{6}\sqrt{\bigg(\dfrac{1}{4^{s+1}}+\dfrac{3}{4}\bigg)(24n -1)}\bigg)}{I_{\frac{3}{2}}\bigg(\dfrac{\pi\sqrt{24n-1}}{6}\bigg)}\sum_{\substack{0< h < 2^{s+1}\\h \textnormal{ odd}}} w(t,h,2^{s+1})e^{-\frac{\pi i nh}{2^s}}.$$
	When $s > 0$ and $n \rightarrow \infty$, the asymptotic behavior of $I_{3/2}$ implies that $$\dfrac{I_{\frac{3}{2}}\bigg(\dfrac{\pi}{6}\sqrt{\bigg(\dfrac{1}{4^{s+1}}+\dfrac{3}{4}\bigg)(24n -1)}\bigg)}{I_{\frac{3}{2}}\bigg(\dfrac{\pi\sqrt{24n-1}}{6}\bigg)} \sim 0.$$
	Therefore when $t$ is even, $\dfrac{A_t(n)}{p(n)} \sim 0$ as $n \rightarrow \infty$. When $s = 0$,  $\dfrac{A_t(n)}{p(n)} \sim (-1)^n 2^{(-t+1)/2}$ as $n \to \infty$.
\end{proof}

\begin{proof}[Proof of Theorem \ref{Even and Odd T Behavior}]
By Proposition \ref{ratios of  A_t(n): p(n)}, we see that
\begin{align*}
\delta_t^e(n) - \delta_t^o(n) \to \begin{cases}
(-1)^n / 2^{(t-1)/2} & \textnormal{if } t \textnormal{ odd}, \\
0 & \textnormal{if } t \textnormal{ even}.
\end{cases}
\end{align*}
Since $\delta_t^e(n) + \delta_t^o(n) = 1$ by definition, the result follows by solving for $\delta_t^e(n)$ and $\delta_t^e(n)$.
\end{proof}

\begin{proof}[Proof of Theorem \ref{distribution property}]
	By Corollary \hyperref[Dominating term of A_t]{\ref{Dominating term of A_t}}, we have $$A_t(n) \sim \displaystyle\dfrac{\pi}{2^{s+\frac{t}{2}}}\bigg(\dfrac{1+ 3\cdot 4^s}{24n -1}\bigg)^\frac{3}{4}I_{\frac{3}{2}}\bigg(\dfrac{\pi\sqrt{(1 + 3\cdot 4^s)(24n -1)}}{6\cdot 2^{s+1}}\bigg)\sum_{\substack{0< h < 2^{s+1}\\h \textnormal{ odd}}}w(t,h,2^{s+1})e^{-\frac{\pi i nh}{2^s}}$$ whose sign is determined by the summation over $h$, which on inspection is periodic in $n$ with period $2^{s+1}$.  In particular, the period is 2 when $s = 0$ which implies the $A_t(n)$ has alternating sign when $t$ is odd as $n \rightarrow \infty$.
\end{proof}

\section{Conclusion}

The surprising nature of this result justifies some reflection. Theorem \ref{Even and Odd T Behavior} differs from the naive expectation of equidistribution in two ways. Not only does equidistribution frequently fail, but there are multiple limiting values when $t$ is odd. Since the distribution properties correspond to the size of $A_t(n)$ in relation to $p(n)$, the proof of Theorem \ref{Exact Formula} reveals on an analytic level the source of these discrepancies. Namely, the $I$-Bessel functions in Theorem \ref{Exact Formula} control whether equidistribution holds and when $t$ is odd the Kloosterman sums arising from $w(t,h,k)$ control the relationship between the parity of $n$ and the sign of $A_t(n)$. All of these details can be read directly off of Theorem \ref{Exact Formula}. However, the circle method does not provide insight into combinatorial explanations of this phenomena, and therefore we leave this question open.

The motivation behind this proof comes from the Nekrasov-Okounkov formula and the applications of this formula made by Han in \cite{Han10} which connect hook numbers to the expansions of various modular forms. In the context of this connection, the problem of the distribution in parity of $\# \mathcal{H}_t(\lambda)$ is translated into a question about asymptotic formulas for the coefficients of a certain modular form. This study has made use of only a microscopic portion of this world of connections, and therefore it is natural to study further problems about $t$-hooks through the lens of modular forms. In particular, in a future paper the first author will study the more difficult question about the distribution of $\# \mathcal{H}_t(\lambda)$ modulo odd primes.

\bibliographystyle{plain}
\bibliography{main}

\end{document}